\documentclass{amsart}
\usepackage{amssymb, amsmath, mathrsfs, verbatim, url, mathabx}
\usepackage[all,cmtip]{xy}

\makeatletter
\@namedef{subjclassname@2020}{%
  \textup{2020} Mathematics Subject Classification}
\makeatother

\theoremstyle{plain}
\newtheorem{theorem}{Theorem}[section]
\newtheorem{lemma}[theorem]{Lemma}
\newtheorem{proposition}[theorem]{Proposition}
\newtheorem{corollary}[theorem]{Corollary}

\theoremstyle{definition}
\newtheorem{definition}[theorem]{Definition}
\newtheorem{question}[theorem]{Question}

\newcommand{\re}{\upharpoonright}

\newcommand{\id}{\mathsf{id}}
\newcommand{\cl}{\mathsf{cl}}
\newcommand{\wc}{\!\!\downarrow\,}
\newcommand{\HC}{\mathsf{HC}}
\newcommand{\Ne}{\mathsf{N}}

\newcommand{\bS}{\mathbf{\Sigma}}
\newcommand{\bP}{\mathbf{\Pi}}
\newcommand{\bD}{\mathbf{\Delta}}
\newcommand{\bG}{\mathbf{\Gamma}}

\newcommand{\bGc}{\widecheck{\mathbf{\Gamma}}}
\newcommand{\bL}{\mathbf{\Lambda}}

\newcommand{\Det}{\mathsf{Det}}
\newcommand{\Borel}{\mathsf{Borel}}
\newcommand{\Diff}{\mathsf{D}}
\newcommand{\Diffc}{\widecheck{\mathsf{D}}}
\newcommand{\VL}{\mathsf{V=L}}
\newcommand{\ZFC}{\mathsf{ZFC}}
\newcommand{\ZF}{\mathsf{ZF}}
\newcommand{\DC}{\mathsf{DC}}
\newcommand{\GL}{\mathsf{L}}
\newcommand{\NSD}{\mathsf{NSD}}
\newcommand{\PU}{\mathsf{PU}}

\newcommand{\AD}{\mathsf{AD}}
\newcommand{\AC}{\mathsf{AC}}
\newcommand{\BP}{\mathsf{BP}}

\newcommand{\BB}{\mathcal{B}}

\newcommand{\XX}{\mathcal{X}}

\newcommand{\UU}{\mathcal{U}}

\newcommand{\PP}{\mathcal{P}}

\newcommand{\RRR}{\mathbb{R}}

\newcommand{\QQQ}{\mathbb{Q}}
\newcommand{\SSS}{\mathbf{S}}
\newcommand{\TTT}{\mathbf{T}}
\newcommand{\cccc}{\mathfrak{c}}

\begin{document}

\title{Zero-dimensional $\sigma$-homogeneous spaces}

\dedicatory{Dedicated to the memory of Ken Kunen}

\author{Andrea Medini}
\address{Institut f\"{u}r Diskrete Mathematik und Geometrie
\newline\indent Technische Universit\"{a}t Wien
\newline\indent  Wiedner Hauptstra\ss e 8–-10/104
\newline\indent 1040 Vienna, Austria}
\email{andrea.medini@tuwien.ac.at}
\urladdr{http://www.dmg.tuwien.ac.at/medini/}

\author{Zolt\'an Vidny\'anszky}
\address{Institute of Mathematics
\newline\indent Faculty of Science, E\"{o}tv\"{o}s Lor\'{a}nd University
\newline\indent P\'{a}zm\'{a}ny P\'{e}ter S\'{e}t\'{a}ny 1/C
\newline\indent H-1117 Budapest, Hungary}
\email{zoltan.vidnyanszky@ttk.elte.hu}
\urladdr{http://vidnyanz.elte.hu}



\thanks{The first-listed author was supported by the FWF grants P 30823, P 35655 and P 35588. The second-listed author was partially supported by the National Research, Development and Innovation Office -- NKFIH grants no.~113047 and no.~129211, and by the FWF grant M 2779. The authors are grateful to Jan van Mill for pointing them to the article \cite{van_engelen_van_mill}, and to the anonymous referee for a careful reading of the paper and several useful suggestions.}

\date{July 17, 2023}

\begin{abstract}
All spaces are assumed to be separable and metrizable. Ostrovsky showed that every zero-dimensional Borel space is $\sigma$-homogeneous. Inspired by this theorem, we obtain the following results:
\begin{itemize}
\item Assuming $\AD$, every zero-dimensional space is $\sigma$-homogeneous,
\item Assuming $\AC$, there exists a zero-dimensional space that is not $\sigma$-homogeneous,
\item Assuming $\VL$, there exists a coanalytic zero-dimensional space that is not $\sigma$-homogeneous.
\end{itemize}
Along the way, we introduce two notions of hereditary rigidity, and give alternative proofs of results of van Engelen, Miller and Steel. It is an open problem whether every analytic zero-dimensional space is $\sigma$-homogeneous.
\end{abstract}

\maketitle

\tableofcontents

\section{Introduction}

Throughout this article, unless we specify otherwise, we will be working in the theory $\ZF+\DC$, that is, the usual axioms of Zermelo-Fraenkel (without the Axiom of Choice) plus the principle of Dependent Choices (see \cite[Section 2]{carroy_medini_muller_2022} for more details and references). By \emph{space} we will always mean separable metrizable topological space.

A space $X$ is \emph{homogeneous} if for every $x,y\in X$ there exists a homeomorphism $h:X\longrightarrow X$ such that $h(x)=y$. For example, using translations, it is easy to see that every topological group is homogeneous (as \cite[Corollary 3.6.6]{van_engelen_thesis} shows, the converse is not true, not even for zero-dimensional Borel spaces). Homogeneity is a classical notion in topology, which has been studied in depth (see for example the survey \cite{arhangelskii_van_mill_2014}).

Here, we will focus on a much less studied notion.\footnote{\,See however \cite{arhangelskii_van_mill_2020_1} and \cite{arhangelskii_van_mill_2020_2}, \cite{van_engelen_van_mill} and \cite{van_mill}, where somewhat related questions are investigated.} We will say that a space $X$ is \emph{$\sigma$-homogeneous} if there exist homogeneous subspaces $X_n$ of $X$ for $n\in\omega$ such that $X=\bigcup_{n\in\omega}X_n$. Homogeneous spaces and countable spaces are trivial examples of $\sigma$-homogeneous spaces. In \cite{ostrovsky}, Alexey Ostrovsky sketched a proof that every zero-dimensional Borel space is $\sigma$-homogeneous. Inspired by his result, we obtained the following theorem (see Theorem \ref{theorem_main} for a stronger result), where $\AD$ denotes the Axiom of Determinacy.

\begin{theorem}\label{theorem_introduction_main}
Assume $\AD$. Then every zero-dimensional space is $\sigma$-homogeneous.
\end{theorem}

To complete the picture, it seemed natural to look for counterexamples under the Axiom of Choice, and for definable counterexamples under $\VL$. This is exactly the content of the following two results (see Corollary \ref{corollary_not_sigma-homogeneous} and Theorem \ref{theorem_definable_not_sigma-homogeneous}), even though the analytic case still eludes us (see Question \ref{question_analytic}).

\begin{theorem}\label{theorem_introduction_counterexample}
There exists a $\ZFC$ example of a zero-dimensional space that is not $\sigma$-homogeneous.
\end{theorem}

\begin{theorem}\label{theorem_introduction_definable_counterexample}
Assume $\VL$. Then there exists a coanalytic zero-dimensional space that is not $\sigma$-homogeneous.
\end{theorem}

While the proof of Theorem \ref{theorem_introduction_counterexample} is a rather straightforward diagonalization, the proof of Theorem \ref{theorem_introduction_definable_counterexample} uses a method developed by the second-listed author in \cite{vidnyanszky}. This method is a ``black-box'' version of the technique that is mostly known by the applications given by Miller in \cite{miller}, and has spawned many more since then. It seems particularly fitting for this special issue to mention that the very first instance of this idea seems to have appeared in a paper by Erd\H{o}s, Kunen and Mauldin (see \cite[Theorems 13, 14 and 16]{erdos_kunen_mauldin}).

Regarding Theorem \ref{theorem_introduction_main}, our fundamental tool will be Wadge theory, which was founded by William Wadge in his doctoral thesis \cite{wadge_thesis} (see also \cite{wadge_2012}), and has become a classical topic in descriptive set theory. The application of Wadge theory to topology was pioneered by Fons van Engelen, based on the fine analysis of the Borel Wadge classes given by Alain Louveau in \cite{louveau_1983}. In particular, in his remarkable doctoral thesis \cite{van_engelen_thesis} (see also \cite{van_engelen_1986}), van Engelen gave a complete classification of the homogeneous zero-dimensional Borel spaces. Other articles in this vein are \cite{van_engelen_1994} and \cite{medini_2019}, where purely topological characterizations of Borel filters and Borel semifilters on $\omega$ are given.\footnote{\,There, by identifying subsets of $\omega$ with their characteristic functions, filters and semifilters on $\omega$ are viewed as subspaces of $2^\omega$.}

Notice that, as the analysis in \cite{louveau_1983} is limited to the Borel context, the same limitation holds for all the results mentioned above. In fact, apart from the present article, the only applications of Wadge theory to topology that go beyond the Borel realm are given in \cite{van_engelen_miller_steel} and \cite{carroy_medini_muller_2020}. In \cite{van_engelen_miller_steel}, the authors take a very game-oriented approach (using Lipschitz-reduction and the machinery of filling in), while \cite{carroy_medini_muller_2020} gives a more systematic structural analysis of the Wadge classes generated by zero-dimensional homogeneous spaces, based on ideas from Louveau's unpublished book \cite{louveau_book}. Here, we will proceed in the spirit of the latter article. In fact, our main reference will be \cite{carroy_medini_muller_2022}, where the Wadge-theoretic portion of \cite{louveau_book} is generalized, and presented in full detail.

Something that \emph{all} topological applications of Wadge theory have in common (and this article is no exception) is that they ultimately rely on a theorem of Steel from \cite{steel_1980}, which is discussed in Section \ref{section_steel}. This theorem, loosely speaking, allows one to make the very considerable leap from Wadge equivalence to homeomorphism.

We conclude this section by recalling two more topological notions that will appear naturally in the course of our investigations. A space $X$ is \emph{rigid} if $|X|\geq 2$ and the only homeomorphism $h:X\longrightarrow X$ is the identity.\footnote{\,The requirement that $|X|\geq 2$ is not standard, but it will enable us to avoid trivialities, thus simplifying the statement of several results.} See \cite[Appendix 2]{van_douwen} and \cite{medini_van_mill_zdomskyy_2016} for an introduction to rigid spaces.

A space $X$ is \emph{strongly homogeneous} (or \emph{h-homogeneous}) if every non-empty clopen subspace of $X$ is homeomorphic to $X$. This notion has been studied by several authors, both ``instrumentally'' and for its own sake (see the list of references in \cite{medini_2011}). While it is well-known that every zero-dimensional strongly homogeneous space is homogeneous (see for example \cite[1.9.1]{van_engelen_thesis} or \cite[Proposition 3.32]{medini_thesis}), the other implication depends on set-theoretic assumptions (see \cite[Theorems 1.1 and 1.2]{carroy_medini_muller_2020}).

\section{Preliminaries and notation}

Let $Z$ be a set, and let $\bG\subseteq\PP(Z)$. Define $\bGc=\{Z\setminus A:A\in\bG\}$. We will say that $\bG$ is \emph{selfdual} if $\bG=\bGc$. Also define $\Delta(\bG)=\bG\cap\bGc$. Given a function $f:Z\longrightarrow W$, $A\subseteq Z$ and $B\subseteq W$, we will use the notation $f[A]=\{f(x):x\in A\}$ and $f^{-1}[B]=\{x\in Z:f(x)\in B\}$. We will use $\id_X:X\longrightarrow X$ to denote the identity function on a set $X$.

\begin{definition}[Wadge]
Let $Z$ be a space, and let $A,B\subseteq Z$. We will write $A\leq B$ if there exists a continuous function $f:Z\longrightarrow Z$ such that $A=f^{-1}[B]$.\footnote{\,Wadge-reduction is usually denoted by $\leq_\mathsf{W}$, which allows to distinguish it from other types of reduction (such as Lipschitz-reduction). Since we will not consider any other type of reduction in this article, we decided to simplify the notation.} In this case, we will say that $A$ is \emph{Wadge-reducible} to $B$. We will write $A<B$ if $A\leq B$ and $B\not\leq A$. We will write $A\equiv B$ if $A\leq B$ and $B\leq A$. We will say that $A$ is \emph{selfdual} if $A\equiv Z\setminus A$.
\end{definition}

\begin{definition}[Wadge]\label{wadgeclassdefinition}
Let $Z$ be a space. Given $A\subseteq Z$, define
$$
A\wc=\{B\subseteq Z:B\leq A\}.
$$
Given $\bG\subseteq\PP(Z)$, we will say that $\bG$ is a \emph{Wadge class} if there exists $A\subseteq Z$ such that $\bG=A\wc$, and that $\bG$ is \emph{continuously closed} if $A\wc\subseteq\bG$ for every $A\in\bG$.\footnote{\,We point out that $A\wc$ is sometimes denoted by $[A]$ (see for example \cite{carroy_medini_muller_2020}, \cite{van_engelen_1986}, \cite{van_engelen_thesis}, \cite{van_engelen_1994} and \cite{louveau_1983}). We decided to avoid this notation, as it conflicts with the notation for the \emph{Wadge degree} of $A$, that is $\{B\subseteq Z:B\equiv A\}$.} We will denote by $\NSD(Z)$ the collection of all non-selfdual Wadge classes in $Z$.
\end{definition}

Both of the above definitions depend of course on the space $Z$. Often, for the sake of clarity, we will specify what the ambient space is by saying, for example, that ``$A\leq B$ in $Z$'' or ``$\bG$ is a Wadge class in $Z$''. Notice that $\{\varnothing\}$ and $\{Z\}$ are trivial examples of Wadge classes in $Z$ for every space $Z$.

Our reference for descriptive set theory is \cite{kechris}. In particular, we assume familiarity with the basic theory of Borel, analytic and coanalytic sets in Polish spaces, and mostly use the same notation as in \cite{kechris}. For example, given a space $Z$, we use $\bS^0_1(Z)$, $\bP^0_1(Z)$, and $\bD^0_1(Z)$ to denote the collection of all open, closed, and clopen subsets of $Z$ respectively. We will denote by $\Borel(Z)$ the collection of all Borel subsets of $Z$. Our reference for other set-theoretic notions is \cite{kunen}.

The classes defined below constitute the so-called \emph{difference hierarchy} (or \emph{small Borel sets}). For a detailed treatment, see \cite[Section 22.E]{kechris} or \cite[Chapter 3]{van_engelen_thesis}.  Recall that, given spaces $X$ and $Z$, a function $j:X\longrightarrow Z$ is an \emph{embedding} if $j:X\longrightarrow j[X]$ is a homeomorphism.

\begin{definition}[Kuratowski]
Let $Z$ be a space, let $1\leq\eta<\omega_1$ and $1\leq\xi<\omega_1$.
\begin{itemize}
\item Given a sequence of sets $(A_\mu:\mu<\eta)$, define
$$
\left.
\begin{array}{lcl}
& & \mathsf{D}_\eta(A_\mu:\mu<\eta)= \left\{
\begin{array}{ll}
\bigcup\{A_\mu\setminus\bigcup_{\zeta<\mu}A_\zeta:\mu<\eta\text{ and }\mu\text{ is odd}\} & \text{if }\eta\text{ is even,}\\
\bigcup\{A_\mu\setminus\bigcup_{\zeta<\mu}A_\zeta:\mu<\eta\text{ and }\mu\text{ is even}\} & \text{if }\eta\text{ is odd.}
\end{array}
\right.
\end{array}
\right.
$$
\item Define $A\in\Diff_\eta(\bS^0_\xi(Z))$ if there exist $A_\mu\in\bS^0_\xi(Z)$ for $\mu<\eta$ such that $A=\Diff_\eta(A_\mu:\mu<\eta)$.\footnote{\,Our notation differs slightly from van Engelen's, as he uses $\Diff_\eta^Z(\bS^0_\xi)$ instead of $\Diff_\eta(\bS^0_\xi(Z))$.} If $\bG=\Diff_\eta(\bS^0_\xi(Z))$, we will denote $\bGc$ by $\Diffc_\eta(\bS^0_\xi(Z))$.
\item A space $X$ is $\Diff_\eta(\bS^0_\xi)$ if $j[X]\in\mathsf{D}_\eta(\bS^0_\xi(Z))$ for every space $Z$ and embedding $j:X\longrightarrow Z$.\footnote{\,These spaces, for obvious reasons, are sometimes called \emph{absolutely $\Diff_\eta(\bS^0_\xi)$}.}
\end{itemize}
\end{definition}

Next, we repeat some useful conventions from \cite{van_engelen_thesis}. Given a topological property $\PP$, we will say that a space $X$ is \emph{nowhere $\PP$} if $X$ is non-empty and no non-empty open subspace of $X$ is $\PP$. For the sake of brevity, we will say that a space is \emph{complete} if it is completely metrizable. We will say that a space $X$ is \emph{strongly $\sigma$-complete} if there exist complete closed subspaces $X_n$ of $X$ for $n\in\omega$ such that $X=\bigcup_{n\in\omega}X_n$.

We will only be interested in $\Diff_\eta(\bS^0_\xi)$ when $\eta\leq\omega$ is even and $\xi=2$. The following characterization is a consequence of \cite[Lemmas 3.1.3, 3.1.4 and 3.1.5]{van_engelen_thesis} and their proofs. Recall that $2\omega=\omega$.
\begin{proposition}[van Engelen]\label{proposition_characterization_differences}
Let $X$ be a space, and let $1\leq k\leq\omega$. Then the following conditions are equivalent:
\begin{itemize}
\item $X$ is $\Diff_{2k}(\bS^0_2)$,
\item $X$ can be written as a countable union of $k$ closed strongly $\sigma$-complete subspaces of $X$,
\item There exist a compact space $Z$ and an embedding $j:X\longrightarrow Z$ such that $j[X]\in\Diff_{2k}(\bS^0_2(Z))$.
\end{itemize}
\end{proposition}

Given a set $A$, we will denote by $A^{<\omega}$ (respectively $A^{\leq\omega}$) the collection of all functions $s:n\longrightarrow A$, where $n<\omega$ (respectively $n\leq\omega$). Given $s\in A^{<\omega}$, we will use the notation $\Ne_s=\{z\in A^\omega:s\subseteq z\}$.\footnote{\,In all our applications, we will have $A=2$ or $A=\omega$.} 

We conclude this section with some miscellaneous topological definitions and results. We will write $X\approx Y$ to mean that the spaces $X$ and $Y$ are homeomorphic. A subset of a space is \emph{clopen} if it is closed and open. A space is \emph{zero-dimensional} if it is non-empty and it has a base consisting of clopen sets.\footnote{\,The empty space has dimension $-1$ (see \cite[Section 7.1]{engelking}).} A space $X$ is a \emph{Borel space} (respectively, \emph{analytic} or \emph{coanalytic}) if there exist a Polish space $Z$ and an embedding $j:X\longrightarrow Z$ such that $j[X]\in\Borel(Z)$ (respectively, $j[X]\in\bS^1_1(Z)$ or $j[X]\in\bP^1_1(Z)$). It is well-known that a space $X$ is Borel (respectively, analytic or coanalytic) iff $j[X]\in\Borel(Z)$ (respectively, $j[X]\in\bS^1_1(Z)$ or $j[X]\in\bP^1_1(Z)$) for every Polish space $Z$ and every embedding $j:X\longrightarrow Z$ (see \cite[Proposition 4.2]{medini_zdomskyy}). For example, by \cite[Theorem 3.11]{kechris}, every Polish space is a Borel space. Given an infinite cardinal $\kappa$, we will say that a space $X$ is \emph{$\kappa$-crowded} if it is non-empty and $|U|=\kappa$ for every non-empty open subset $U$ of $X$. Furthermore, we will say that a subset $X$ of a non-empty space $Z$ is \emph{$\kappa$-dense in $Z$} if $|U\cap X|=\kappa$ for every non-empty open subset $U$ of $Z$.

The next proposition can be safely assumed to be folklore, while Lemma \ref{lemma_terada} follows easily from \cite[Theorem 2.4]{terada} (see also \cite[Theorem 2 and Appendix A]{medini_2011} or \cite[Theorem 3.2 and Appendix B]{medini_thesis}). We will denote by $\BP$ the statement that every subset of a Polish space has the property of Baire. It is well-known that $\AD$ implies $\BP$ (see \cite[Section 3]{carroy_medini_muller_2022}).

\begin{proposition}\label{proposition_baire_comeager}
Assume $\BP$. Let $Z$ be a Polish space, and let $X$ be a dense Baire subspace of $Z$. Then $X$ is comeager in $Z$.
\end{proposition}
\begin{proof}
Since $X$ has the Baire property, we can write $X=G\cup M$ by \cite[Proposition 8.23.ii]{kechris}, where $G\in\bP^0_2(Z)$ and $M$ is meager in $Z$. It will be enough to show that $G$ is dense in $Z$. Assume, in order to get a contradiction, that there exists a non-empty open subset $U$ of $Z$ such that $U\cap G=\varnothing$. Observe that $U\cap X$ is a non-empty open subset of $X$ because $X$ is dense in $Z$. Furthermore, using the density of $X$ again, it is easy to see that $M=M\cap X$ is meager in $X$. Since $U\cap X\subseteq M$, this contradicts the fact that $X$ is a Baire space.
\end{proof}

\begin{lemma}[Terada]\label{lemma_terada}
Let $X$ be a space. Assume that $X$ has a base $\BB\subseteq\bD^0_1(X)$ such that $U\approx X$ for every $U\in\BB$. Then $X$ is strongly homogeneous.
\end{lemma}

\section{The basics of Wadge theory}

The following two results (the first is commonly known as ``Wadge's Lemma'') are the most fundamental theorems of Wadge theory. For the proofs, see \cite[Lemma 3.2 and Theorem 3.3]{carroy_medini_muller_2020}.

\begin{lemma}[Wadge]\label{lemma_wadge}
Assume $\AD$. Let $Z$ be a zero-dimensional Polish space, and let $A,B\subseteq Z$. Then either $A\leq B$ or $Z\setminus B\leq A$.
\end{lemma}

\begin{theorem}[Martin, Monk]\label{theorem_martin_monk}
Assume $\AD$. Let $Z$ be a zero-dimensional Polish space. Then the relation $\leq$ on $\PP(Z)$ is well-founded.
\end{theorem}

As a first application of Wadge's Lemma, one can easily show that the largest antichains with respect to $\leq$ have size $2$. These antichains are in the form $\{\bG,\bGc\}$ for some $\bG\in\NSD(Z)$, and they are known as \emph{non-selfdual pairs}. Another application is given by Lemma \ref{lemma_non-selfdual}, whose simple proof is left to the reader. Proposition \ref{proposition_differences_wadge_class} gives the simplest non-trivial examples of Wadge classes.\footnote{\,A more systematic proof of Proposition \ref{proposition_differences_wadge_class} could be given using Hausdorff operations (see \cite[Sections 5-7]{carroy_medini_muller_2020}). Furthermore, one can easily bypass $\AD$ by using Borel Determinacy.}

\begin{lemma}\label{lemma_non-selfdual}
Assume $\AD$. Let $Z$ be a zero-dimensional Polish space. If $\bG\subseteq\PP(Z)$ is non-selfdual and continuously closed then $\bG\in\NSD(Z)$.
\end{lemma}

\begin{proposition}\label{proposition_differences_wadge_class}
Assume $\AD$. Let $Z$ be an uncountable zero-dimensional Polish space, let $1\leq\eta<\omega_1$ and $1\leq\xi<\omega_1$. Then $\Diff_\eta(\bS^0_\xi(Z))\in\NSD(Z)$.
\end{proposition}
\begin{proof}
Set $\bG=\Diff_\eta(\bS^0_\xi(Z))$. Notice that $\bG$ is continuously closed by \cite[Exercise 22.26.i]{kechris}, while $\bG$ is non-selfdual by \cite[Exercise 22.26.iii]{kechris}. Therefore, the desired conclusion follows from Lemma \ref{lemma_non-selfdual}.
\end{proof}

Observe that non-selfdual Wadge classes seem to suffer from the ``pair of socks'' problem. In less ridiculous words, given a non-selfdual Wadge class $\bG$, there is no obvious way to distinguish $\bG$ from $\bGc$. The following result gives an elegant solution to this problem (see \cite[Theorem 2]{van_wesep} and \cite{steel_1981}). Recall that $\bG\subseteq\PP(\omega^\omega)$ has the \emph{separation property} if whenever $A$ and $B$ are disjoint elements of $\bG$ there exists $C\in\Delta(\bG)$ such that $A\subseteq C\subseteq \omega^\omega\setminus B$.
\begin{theorem}[Steel, Van Wesep]\label{theorem_separation}
Assume $\AD$. For every $\bG\in\NSD(\omega^\omega)$, exactly one of $\bG$ and $\bGc$ has the separation property.
\end{theorem}
The way in which we will apply the above theorem is rather amusing, as we will not care at all about what the separation property says specifically: we will only use the fact that it provides a canonical choice between $\bG$ and $\bGc$, thus allowing us to bypass $\AC$ (see the definition of $\HC(X)$ at the beginning of Section \ref{section_main}).\footnote{\,Recall that we are working in $\ZF+\DC$.}

Next, we give the well-known analysis of the selfdual sets. Using Theorem \ref{theorem_selfdual}, it is often possible to reduce the selfdual case to the non-selfdual case. For the proof, see \cite[Corollary 5.5]{carroy_medini_muller_2022}.

\begin{theorem}\label{theorem_selfdual}
Assume $\AD$. Let $Z$ be a zero-dimensional Polish space, let $V\in\bD^0_1(Z)$, and let $A$ be a selfdual subset of $Z$. Then there exist pairwise disjoint $V_n\in\bD^0_1(V)$ and non-selfdual $A_n<A$ in $Z$ for $n\in\omega$ such that $\bigcup_{n\in\omega}V_n=V$ and $\bigcup_{n\in\omega}(A_n\cap V_n)=A\cap V$.
\end{theorem}

We conclude this section with two closure properties. The elementary proof of Proposition \ref{proposition_closure_clopen} is left to the reader. The case $Z=\omega^\omega$ of Theorem \ref{theorem_closure_closed} is due to Andretta, Hjorth and Neeman (see \cite[Lemma 3.6.a]{andretta_hjorth_neeman}), however we refer to \cite[Lemma 12.3]{carroy_medini_muller_2020} for the proof of the more general version given here.

\begin{proposition}\label{proposition_closure_clopen}
Let $Z$ be a space, let $\bG$ be a Wadge class in $Z$, and let $A\in\bG$.
\begin{itemize}
\item Assume that $\bG\neq\{Z\}$. Then $A\cap V\in\bG$ for every $V\in\bD^0_1(Z)$.
\item Assume that $\bG\neq\{\varnothing\}$. Then $A\cup V\in\bG$ for every $V\in\bD^0_1(Z)$.
\end{itemize}
\end{proposition}

\begin{theorem}\label{theorem_closure_closed}
Assume $\AD$. Let $Z$ be an uncountable zero-dimensional Polish space, and let $\bG\in\NSD(Z)$. Assume that $\mathsf{D}_n(\bS^0_1(Z))\subseteq\bG$ whenever $1\leq n<\omega$.
	\begin{itemize}
	\item If $A\in\bG$ and $C\in\bP^0_1(Z)$ then $A\cap C\in\bG$.
	\item If $A\in\bG$ and $U\in\bS^0_1(Z)$ then $A\cup U\in\bG$.
	\end{itemize}
\end{theorem}

\section{Relativization}\label{section_relativization}

In this section we will set up the machinery of relativization, which will allow us to consider the ``same'' Wadge class in different ambient spaces. The results in this section are essentially due to Louveau and Saint-Raymond (see \cite[Section 4]{louveau_saint-raymond_1988_2}), although a more systematic exposition was given in \cite[Sections 6 and 7]{carroy_medini_muller_2022}.\footnote{\,An alternative approach to relativization (although equivalent in practice) was given in \cite[Section 6]{carroy_medini_muller_2020}, using Hausdorff operations. The approach given here has the advantage of not having to rely on Van Wesep's theorem (see \cite[Section 8]{carroy_medini_muller_2020}), which is far from elementary.} We refer to the latter for the proofs and a more extensive treatment.

Definition \ref{definition_relativization} and Lemma \ref{lemma_relativization_exists_unique} establish the foundations of this method, while Lemmas \ref{lemma_relativization} and \ref{lemma_relativization_subspace} give several ``reassuring'' and extremely useful properties of relativization.

\begin{definition}[Louveau, Saint-Raymond]\label{definition_relativization}
Given a space $Z$ and $\bG\subseteq\PP(\omega^\omega)$, define
$$
\bG(Z)=\{A\subseteq Z:g^{-1}[A]\in\bG\text{ for every continuous }g:\omega^\omega\longrightarrow Z\}.
$$	
\end{definition}

\begin{lemma}\label{lemma_relativization_exists_unique}
Assume $\AD$. Let $Z$ be a zero-dimensional Polish space, and let $\bL\in\NSD(Z)$. Then there exists a unique $\bG\in\NSD(\omega^\omega)$ such that $\bG(Z)=\bL$.
\end{lemma}

\begin{lemma}\label{lemma_relativization}
Let $\bG\subseteq\PP(\omega^\omega)$, and let $Z$ and $W$ be spaces.
\begin{itemize}
\item If $f:Z\longrightarrow W$ is continuous and $B\in\bG(W)$ then $f^{-1}[B]\in\bG(Z)$.
\item If $h:Z\longrightarrow W$ is a homeomorphism then $A\in\bG(Z)$ iff $h[A]\in\bG(W)$.
\item $\widecheck{\bG(Z)}=\bGc(Z)$.
\item If $\bG$ is continuously closed then $\bG(\omega^\omega)=\bG$.
\end{itemize}
\end{lemma}

\begin{lemma}\label{lemma_relativization_subspace}
Assume $\AD$. Let $Z$ and $W$ be zero-dimensional Borel spaces such that $W\subseteq Z$, and let $\bG\in\NSD(\omega^\omega)$. Then $A\in\bG(W)$ iff $A=\widetilde{A}\cap W$ for some $\widetilde{A}\in\bG(Z)$.
\end{lemma}

Next, we state special cases of \cite[Theorems 7.1 and 7.2]{carroy_medini_muller_2022}. These results show that, when the ambient spaces $Z$ and $W$ are zero-dimensional, uncountable and Polish, the correspondence $\bG(Z)\longmapsto\bG(W)$ is an order-isomorphism between the non-selfdual Wadge classes in $Z$ and the non-selfdual Wadge classes in $W$.

\begin{theorem}\label{theorem_order_isomorphism}
Assume $\AD$. Let $Z$ and $W$ be zero-dimensional uncountable Borel spaces, and let $\bG,\bL\in\NSD(\omega^\omega)$. Then
$$
\bG(Z)\subseteq\bL(Z)\text{ iff }\bG(W)\subseteq\bL(W).
$$
\end{theorem}

\begin{theorem}\label{theorem_van_wesep_surrogate}
Assume $\AD$. Let $Z$ be a zero-dimensional uncountable Polish space. Then
$$
\NSD(Z)=\{\bG(Z):\bG\in\NSD(\omega^\omega)\}.
$$
\end{theorem}

\section{The notion of level}\label{section_level}

The aim of this section it to introduce an extremely useful tool in the analysis of non-selfdual Wadge classes. Essentially, this notion first appeared in \cite{louveau_1983}, but we will follow the approach of \cite{louveau_saint-raymond_1988_1}, which was however limited to the Borel context (see also \cite[Section 7.3.4]{louveau_book}). We begin with a preliminary definition.\footnote{\,In \cite{louveau_saint-raymond_1988_1}, the notation $\bD_{1+\xi}^0\text{-}\PU$ is used instead of $\PU_\xi$, and $\lambda_\mathsf{C}$ is used instead of $\ell$.}

\begin{definition}
Let $Z$ be a space, let $\bG\subseteq\PP(Z)$, and let $\xi<\omega_1$. Define $\PU_\xi(\bG)$ to be the collection of all sets of the form
$$
\bigcup_{n\in\omega}(A_n\cap V_n),
$$
where each $A_n\in\bG$, the $V_n\in\bD_{1+\xi}^0(Z)$ are pairwise disjoint, and $\bigcup_{n\in\omega}V_n=Z$. Sets of this form are known as \emph{partitioned unions} of sets in $\bG$.
\end{definition}

\begin{definition}[Louveau, Saint-Raymond]
Let $Z$ be a space, let $\bG\subseteq\PP(Z)$, and let $\xi<\omega_1$. Define
\begin{itemize}
\item $\ell(\bG)\geq\xi$ if $\PU_\xi(\bG)=\bG$,
\item $\ell(\bG)=\xi$ if $\ell(\bG)\geq\xi$ and $\ell(\bG)\not\geq\xi+1$,
\item $\ell(\bG)=\omega_1$ if $\ell(\bG)\geq\eta$ for every $\eta<\omega_1$.
\end{itemize}
We refer to $\ell(\bG)$ as the \emph{level} of $\bG$.
\end{definition}

We remark that the notion of level is better understood in the context of the fundamental Expansion Theorem (see \cite[Theorem 16.1]{carroy_medini_muller_2022}), and that it is possible (although far from easy) to show that for every non-selfdual Wadge $\bG$ there exists $\xi\leq\omega_1$ such that $\ell(\bG)=\xi$ (see \cite[Corollary 17.2]{carroy_medini_muller_2022}). However, since we will not need these results, we will not say anything more about them.

The following technical result will be needed in the proof of Lemma \ref{lemma_homogeneous_clopen}. It is simply a restatement of \cite[Lemma 14.3]{carroy_medini_muller_2020}.

\begin{lemma}\label{lemma_level_0}
Assume $\AD$. Let $Z$ be an uncountable zero-dimensional Polish space, let $\bG\in\NSD(Z)$ be such that $\ell(\bG)=0$, and let $X\in\bG$ be codense in $Z$. Then there exist a non-empty $U\in\bD^0_1(Z)$ and $\bL\in\NSD(Z)$ such that $\bL\subsetneq\bG$ and $X\cap U\in\bL$.	
\end{lemma}

\section{Steel's theorem and good Wadge classes}\label{section_steel}

As we mentioned in the introduction, the following theorem (which is a particular case of \cite[Theorem 2]{steel_1980}) is one of our main tools. Given a Wadge class $\bG$ in $2^\omega$ and $X\subseteq 2^\omega$, we will say that $X$ is \emph{everywhere properly $\bG$} if $X\cap\Ne_s\in\bG\setminus\bGc$ for every $s\in 2^{<\omega}$. We will not give the definition of reasonably closed class, as it is an \emph{ad hoc} notion that would not be particularly enlightening. In fact, the only property of reasonably closed classes that we will need is the one given by Lemma \ref{lemma_good_implies_reasonably_closed}. We refer the interested reader to \cite[Section 13]{carroy_medini_muller_2020}.

\begin{theorem}[Steel]\label{theorem_steel}
Assume $\AD$. Let $\bG$ be a reasonably closed Wadge class in~$2^\omega$. Assume that $X$ and $Y$ are subsets of $2^\omega$ that satisfy the following conditions:
\begin{itemize}
\item $X$ and $Y$ are everywhere properly $\bG$,
\item $X$ and $Y$ are either both meager in $2^\omega$ or both comeager in $2^\omega$.
\end{itemize}
Then there exists a homeomorphism $h:2^\omega\longrightarrow 2^\omega$ such that $h[X]=Y$.
\end{theorem}

The following notion first appeared as \cite[Definition 12.1]{carroy_medini_muller_2020}, inspired by work of van Engelen. For a proof of Lemma \ref{lemma_good_implies_reasonably_closed}, see \cite[Lemma 13.2]{carroy_medini_muller_2020}.

\begin{definition}[Carroy, Medini, M\"uller]\label{definition_good}
Let $Z$ be a space, and let $\bG$ be a Wadge class in $Z$. We will say that $\bG$ is \emph{good} if the following conditions are satisfied:
\begin{itemize}
\item $\bG$ is non-selfdual,
\item $\Delta(\mathsf{D}_\omega(\bS^0_2(Z)))\subseteq\bG$,
\item $\ell(\bG)\geq 1$.
\end{itemize}
\end{definition}

\begin{lemma}[Carroy, Medini, M\"uller]\label{lemma_good_implies_reasonably_closed}
Assume $\AD$. Let $\bG$ be a good Wadge class in $2^\omega$. Then $\bG$ is reasonably closed.
\end{lemma}

\section{Zero-dimensional homogeneous spaces of low complexity}\label{section_low_complexity}

In his thesis \cite{van_engelen_thesis}, van Engelen discovered that the behavior of zero-dimensional homogeneous spaces changes drastically depending on whether their complexity is higher or lower than $\bD=\Delta(\mathsf{D}_\omega(\bS^0_2))$. While Wadge theory is the perfect tool for dealing with spaces of complexity above $\bD$, different techniques are needed for spaces of complexity below $\bD$ (and these techniques work well in the more general setting of $\bD^0_2$ spaces).

We will essentially follow the same subdivision into cases here, except that in the proof of Lemma \ref{lemma_homogeneous_clopen}, it will be convenient to use $\Diff_\omega(\bS^0_2)$ as the dividing line (this is the reason why below we will consider $\PP_\omega$ and $\XX_\omega$). We begin by defining a collection of relevant topological properties, taken from \cite[Definition 3.1.7]{van_engelen_thesis} (see also \cite[Lemma 3.1.4]{van_engelen_thesis}) and \cite[Definition 3.1.8]{van_engelen_thesis}. For convenience, we will stipulate that $\Diff_0(\bS^0_2)$ is the property of being the empty space.

\begin{definition}[van Engelen]
Given a space $X$ and $k\in\omega$, we will say that:
\begin{itemize}
\item $X$ is $\PP_{4k}$ if $X$ is the union of a $\Diff_{2k}(\bS^0_2)$ subspace and a complete subspace,
\item $X$ is $\PP_{4k+1}$ if $X$ is $\Diff_{2(k+1)}(\bS^0_2)$,
\item $X$ is $\PP^1_{4k+2}$ if $X$ is the union of a $\Diff_{2k}(\bS^0_2)$ subspace, a complete subspace, and a countable subspace,
\item $X$ is $\PP^1_{4k+3}$ if $X$ is the union of a $\Diff_{2(k+1)}(\bS^0_2)$ subspace and a countable subspace,
\item $X$ is $\PP^2_{4k+2}$ if $X$ is the union of a $\Diff_{2k}(\bS^0_2)$ subspace, a complete subspace, and a $\sigma$-compact subspace,
\item $X$ is $\PP^2_{4k+3}$ if $X$ is the union of a $\Diff_{2(k+1)}(\bS^0_2)$ subspace and a $\sigma$-compact subspace.
\end{itemize}
It will also be useful to define the following:
\begin{itemize}
\item $X$ is $\PP^1_{-2}$ if $X$ has size at most $1$,
\item $X$ is $\PP^1_{-1}$ if $X$ is countable,
\item $X$ is $\PP^2_{-2}$ if $X$ is compact,
\item $X$ is $\PP^2_{-1}$ if $X$ is $\sigma$-compact,
\item $X$ is $\PP_\omega$ if $X$ is $\Diff_\omega(\bS^0_2)$.
\end{itemize}
To indicate one of these properties generically (that is, in case we do not know whether the superscript $i\in\{1,2\}$ is present or not) we will use the notation $\PP_n^{(i)}$.
\end{definition}

Declare a linear order $\prec$ on these properties as follows:
\begin{multline}
\PP_{-2}^1\prec\PP_{-1}^1\prec\PP_{-2}^2\prec\PP_{-1}^2\prec\cdots\\\nonumber
\cdots\prec\PP_{4k}\prec\PP_{4k+1}\prec\PP^1_{4k+2}\prec\PP^1_{4k+3}\prec\PP^2_{4k+2}\prec\PP^2_{4k+3}\prec\cdots\prec\PP_\omega.
\end{multline}

For every $n\in\{-2,-1\}\cup\omega\cup\{\omega\}$ and $i\in\{1,2\}$ such that $\PP_n^{(i)}$ is defined, we will consider a class of spaces $\XX_n^{(i)}$, and use the same notational convention. Instead of giving their definitions, we will use the characterization given by the following theorem (see \cite[Theorem 3.4.24]{van_engelen_thesis} for the case $n<\omega$, and \cite[Definition 3.5.7]{van_engelen_thesis} for the case $n=\omega$). In fact, this characterization is at the same time more understandable and more suitable for the applications given here.

\begin{theorem}[van Engelen]\label{theorem_characterization}
Let $n\in\{-2,-1\}\cup\omega\cup\{\omega\}$ and $i\in\{1,2\}$. Then, for a zero-dimensional space $X$, the following conditions are equivalent:
\begin{itemize}
\item $X\in\XX_n^{(i)}$,
\item $X$ is $\PP_n^{(i)}$ and nowhere $\PP_m^{(j)}$ for every $m\in\{-2,-1\}\cup\omega\cup\{\omega\}$ and $j\in\{1,2\}$ such that $\PP_m^{(j)}\prec\PP_n^{(i)}$.
\end{itemize}
\end{theorem}

The fundamental property of the classes $\XX_n^{(i)}$ that we will need is given by the following result (see \cite[Theorem 3.4.13]{van_engelen_thesis} for the case $n<\omega$ and \cite[Theorem 3.5.9]{van_engelen_thesis} for the case $n=\omega$).

\begin{theorem}[van Engelen]\label{theorem_exactly_one}
Let $n\in\{-2,-1\}\cup\omega\cup\{\omega\}$ and $i\in\{1,2\}$. Then, up to homeomorphism, the class $\XX_n^{(i)}$ contains exactly one element, which is strongly homogeneous.
\end{theorem}

A few clarifications are in order. First of all, van Engelen did not define $\PP_n^{(i)}$ or $\XX_n^{(i)}$ for $n=-2$. In particular, he did not include them in his statements of the above two theorems. However, it is not hard to realize that they naturally fit into the context described here. In fact, using the classical characterizations of $2^\omega$, $\QQQ$ and $\QQQ\times 2^\omega$ (see \cite[Theorems 2.1.1, 2.4.1 and 2.4.5]{van_engelen_thesis}), one sees that:
\begin{itemize}
\item $\XX^1_{-2}$ is the class of spaces of size $1$,
\item $\XX^1_{-1}$ is the class of spaces that are homeomorphic to $\QQQ$,
\item $\XX^2_{-2}$ is the class of spaces that are homeomorphic to $2^\omega$,
\item $\XX^2_{-1}$ is the class of spaces that are homeomorphic to $\QQQ\times 2^\omega$.
\end{itemize}
Furthermore, the same characterizations show that Theorem \ref{theorem_exactly_one} holds for these classes as well. Finally, we remark that the class $\XX_\omega$ is denoted $\XX_\omega^2$ by van Engelen, but here we prefer to make sure that the notation for each class $\XX_n^{(i)}$ always matches the one for the corresponding property $\PP_n^{(i)}$.

The following diagram (which is taken from \cite[page 28]{van_engelen_thesis}) illustrates the first few classes $\XX_n^{(i)}$. For a concrete description of the spaces $\TTT$ and $\SSS$ (introduced by van Douwen and van Mill respectively), see \cite[Section 5]{medini_2019}.

\begin{center}
$
\xymatrix{
1\in\XX_{-2}^1 \ar@{-}[d] & & 2^\omega\in\XX_{-2}^2 \ar@{-}[d]\\
\QQQ\in\XX_{-1}^1 \ar@{-}@/_/[rd] & & \QQQ\times 2^\omega\in\XX_{-1}^2 \ar@{-}@/^/[ld]\\
& \omega^\omega\in\XX_0 \ar@{-}[d] &\\
& \QQQ\times\omega^\omega\in\XX_1 \ar@{-}@/_/[ld] \ar@{-}@/^/[rd] &\\
\TTT\in\XX_2^1 \ar@{-}[d] & & \SSS\in\XX_2^2 \ar@{-}[d]\\
\QQQ\times\TTT\in\XX_3^1 \ar@{-}@/_/[rd] & & \QQQ\times\SSS\in\XX_3^2 \ar@{-}@/^/[ld]\\
& \XX_4 \ar@{-}[d] &\\
& \vdots &\\
}
$
\end{center}

\section{The positive result}\label{section_main}

In this section, we will put together all the tools accumulated so far to obtain the positive result announced in the introduction (namely, Theorem \ref{theorem_main}). We begin by defining, given a subspace $X$ of $2^\omega$, a subspace $\HC(X)$ of $X$ which is $\mathsf{H}$omogeneous (by Lemma \ref{lemma_homogeneous_clopen}) and $\mathsf{C}$lopen (by construction). Throughout this section, we will assume that an enumeration $\bD^0_1(2^\omega)=\{C_k:k\in\omega\}$ of the clopen subsets of $2^\omega$ has been fixed. We remark that $\HC(X)$ is not canonical, as it will depend on this enumeration.

\noindent{\bf Case 1:} $X=\varnothing$.

In this case, simply set $\HC(X)=\varnothing$.

\noindent{\bf Case 2:} $X$ has a non-empty open $\Diff_\omega(\bS^0_2)$ subspace.

Let $\PP$ be the $\prec$-minimal\footnote{\,The ordering $\prec$ is described in Section \ref{section_low_complexity}.} property $\PP_n^{(i)}$, where $n\in\{-2,-1\}\cup\omega\cup\{\omega\}$ and $i\in\{1,2\}$, such that there exists $C\in\bD^0_1(2^\omega)$ satisfying the following conditions:
\begin{itemize}
\item $X\cap C\neq\varnothing$,
\item $X\cap C$ is $\PP$.
\end{itemize}
Now fix $C=C_k$, where $k\in\omega$ is minimal such that the two conditions above are satisfied. Finally, define
$$
\HC(X)=X\cap C.
$$

\noindent{\bf Case 3:} $X$ is nowhere $\Diff_\omega(\bS^0_2)$.

Fix a $\subseteq$-minimal $\bG\in\NSD(\omega^\omega)$ such that there exists $C\in\bD^0_1(2^\omega)$ satisfying the following conditions:
\begin{itemize}
\item $X\cap C\neq\varnothing$,
\item $X\cap C\in\bG(2^\omega)$.
\end{itemize}
In case there are two possible choices for $\bG$ (which would have to be a non-selfdual pair by Lemma \ref{lemma_wadge}), pick the one that has the separation property (recall Theorem \ref{theorem_separation}). Now fix $C=C_k$, where $k\in\omega$ is minimal such that the two conditions above plus the following one are satisfied:
\begin{itemize}
\item Either $X\cap C$ is a Baire space or $X\cap C$ is a meager space.
\end{itemize}
Finally, define
$$
\HC(X)=X\cap C.
$$

\begin{lemma}\label{lemma_homogeneous_clopen}
Assume $\AD$. Let $X$ be a subspace of $2^\omega$. Then $\HC(X)$ is a strongly homogeneous clopen subspace of $X$.
\end{lemma}
\begin{proof}
Set $U=\HC(X)$. The fact that $U$ is clopen in $X$ is clear. To see that $U$ is strongly homogeneous, we will consider the three cases that appear in the definition of $\HC(X)$.

\noindent{\bf Case 1:} $X=\varnothing$.

The desired conclusion holds trivially in this case.

\noindent{\bf Case 2:} $X$ has a non-empty open $\Diff_\omega(\bS^0_2)$ subspace.

Let $\PP$ be as in the definition of $\HC(X)$. Using the minimality of $\PP$ and the fact that each $\PP^{(i)}_n$ is inherited by clopen subspaces, it is straightforward to check that $U$ is $\PP$ and nowhere $\PP_m^{(j)}$ for every $m\in\{-2,-1\}\cup\omega\cup\{\omega\}$ and $j\in\{1,2\}$ such that $\PP_m^{(j)}<\PP$. By Theorems \ref{theorem_characterization} and \ref{theorem_exactly_one}, it follows that $U$ is strongly homogeneous.

\noindent{\bf Case 3:} $X$ is nowhere $\Diff_\omega(\bS^0_2)$.

Let $C$ and $\bG$ be as in the definition of $\HC(X)$. Throughout the rest of this proof, we will write $\cl$ to denote closure in $2^\omega$. Let $K=\cl(U)$, and observe that $K\approx 2^\omega$.

\noindent{\bf Claim 1.} If $\bL\in\NSD(\omega^\omega)$ and there exists $V\in\bD^0_1(2^\omega)$ such that $X\cap V\neq\varnothing$ and $X\cap V\in\bL(2^\omega)$, then $\Diffc_\omega(\bS^0_2(2^\omega))\subseteq\bL(2^\omega)$.

Let $\bL$ and $V$ be as in the statement of the claim. Notice that $X\cap V\notin\Diff_\omega(\bS^0_2(2^\omega))$ by Proposition \ref{proposition_characterization_differences}, since $X$ is nowhere $\Diff_\omega(\bS^0_2)$. Since $\Diff_\omega(\bS^0_2(2^\omega))$ is a Wadge class by Proposition \ref{proposition_differences_wadge_class}, the desired result follows from Lemma \ref{lemma_wadge}. $\blacksquare$

\noindent{\bf Claim 2.} $U\cap V$ is non-selfdual in $2^\omega$ for every $V\in\bD^0_1(2^\omega)$.

Pick $V\in\bD^0_1(2^\omega)$. The claim is trivial if $U\cap V=\varnothing$, so assume that $U\cap V\neq\varnothing$. Assume, in order to get a contradiction, that $U\cap V$ is selfdual. By Theorem \ref{theorem_selfdual}, we can fix pairwise disjoint $V_n\in\bD^0_1(2^\omega)$ and non-selfdual $A_n<U\cap V$ for $n\in\omega$ such that $\bigcup_{n\in\omega}V_n=2^\omega$ and $\bigcup_{n\in\omega}(A_n\cap V_n)=U\cap V$. Fix $n$ such that $A_n\cap V_n\neq\varnothing$. Notice that $A_n\neq 2^\omega$ because $X$ is nowhere $\Diff_\omega(\bS^0_2)$, hence $A_n\cap V_n\leq A_n < U\cap V\in\bG(2^\omega)$ by Proposition \ref{proposition_closure_clopen}. Since $A_n\cap V_n=U\cap V\cap V_n=X\cap C\cap V\cap V_n$, this contradicts the minimality of $\bG$. $\blacksquare$

\noindent{\bf Claim 3.} $\bG(2^\omega)$ is a good Wadge class.

The fact that $\bG(2^\omega)$ is non-selfdual follows from Theorem \ref{theorem_van_wesep_surrogate}, as $\bG$ is non-selfdual by construction. Furthermore, one sees from Claim 1 that $\Delta(\Diff_\omega(\bS^0_2(2^\omega)))\subseteq\bG(2^\omega)$. It remains to show that $\ell(\bG(2^\omega))\geq 1$. Since $K\approx 2^\omega$, it will be enough to show that $\ell(\bG(K))\geq 1$. So assume, in order to get a contradiction, that $\ell(\bG(K))=0$. Notice that $U$ is codense in $K$ because $X$ is nowhere $\Diff_\omega(\bS^0_2)$. Therefore, it is possible to apply Lemma \ref{lemma_level_0}. Together with Lemma \ref{lemma_relativization_exists_unique} and Theorem \ref{theorem_order_isomorphism}, this yields a non-empty $V\in\bD^0_1(K)$ and $\bL\in\NSD(\omega^\omega)$ such that $\bL\subsetneq\bG$ and $U\cap V\in\bL(K)$. Set $W=U\cap V$, and notice that $W\neq\varnothing$ because $U$ is dense in $K$. By Lemma \ref{lemma_relativization_subspace}, there exists $\widetilde{W}\in\bL(2^\omega)$ such that $\widetilde{W}\cap K=W$. Furthermore, we must have $\bS^0_2(K)\subseteq\bL(K)$, otherwise we would have $W\in\bP^0_2(K)$ by Lemma \ref{lemma_wadge}, contradicting the fact that $X$ is nowhere $\Diff_\omega(\bS^0_2)$. Since $K\approx 2^\omega$, it follows that $\bS^0_2(2^\omega)\subseteq\bL(2^\omega)$. In particular, it is possible to apply Theorem \ref{theorem_closure_closed}, which yields $W=\widetilde{W}\cap K\in\bL(2^\omega)$. Finally, using the compactness of $K$, it is easy to see that $W$ is in the form $X\cap D$ for some $D\in\bD^0_1(2^\omega)$, which contradicts the minimality of $\bG$. $\blacksquare$

Finally, thanks to Claim 3 and Lemma \ref{lemma_good_implies_reasonably_closed}, we are in a position to apply Theorem \ref{theorem_steel} to obtain the strong homogeneity of $U$. By Lemma \ref{lemma_terada}, it will be enough to show that $D\cap U\approx U$ for every $D\in\bD^0_1(2^\omega)$ such that $D\cap U\neq\varnothing$. So fix such a $D$, set $L=\cl(D\cap U)$, and observe that $L\approx 2^\omega$. Recall that $U$ is either a Baire space or a meager space by construction. In the first case, Proposition \ref{proposition_baire_comeager} shows that $U$ is a comeager subset of $K$, hence $D\cap U$ is a comeager subset of $L$. In the second case, it is clear that $U$ is a meager subset of $K$, and that $D\cap U$ is a meager subset of $L$. Therefore, the next two claims will conclude the proof.

\noindent{\bf Claim 4.} $U$ is everywhere properly $\bG(K)$.

Let $V\in\bD^0_1(2^\omega)$ be such that $U\cap V\neq\varnothing$, and set $W=U\cap V$. Notice that $W\in\bG(2^\omega)$ by Proposition \ref{proposition_closure_clopen}, hence $W\in\bG(K)$ by Lemma \ref{lemma_relativization_subspace}. Now assume, in order to get a contradiction, that $W\in\bGc(K)$. By Lemma \ref{lemma_relativization_subspace}, we can fix $\widetilde{W}\in\bGc(2^\omega)$ such that $\widetilde{W}\cap K=W$. By Claim 1, it is possible to apply Theorem \ref{theorem_closure_closed}, which yields $W=\widetilde{W}\cap K\in\bGc(2^\omega)$. Since $W$ is non-selfdual in $2^\omega$ by Claim 2, this contradicts the minimality of $\bG$. $\blacksquare$

\noindent{\bf Claim 5.} $D\cap U$ is everywhere properly $\bG(L)$.

The proof of this claim is similar to the proof of Claim 4. $\blacksquare$
\end{proof}

\begin{theorem}\label{theorem_main}
Assume $\AD$. Then every zero-dimensional space is $\sigma$-homogeneous. More precisely, every zero-dimensional space $X$ can be written as a countable disjoint union of closed strongly homogeneous subspaces of $X$.
\end{theorem}
\begin{proof}
Let $X$ be a zero-dimensional space. Without loss of generality, assume that $X$ is a subspace of $2^\omega$. Using transfinite recursion, define $X_\alpha$ for every ordinal $\alpha$ as follows:
\begin{itemize}
\item $X_0=X$,
\item $X_{\alpha+1}=X_\alpha\setminus\HC(X_\alpha)$,
\item $X_\gamma=\bigcap_{\alpha<\gamma} X_\alpha$, if $\gamma$ is a limit ordinal.
\end{itemize}
Since the $X_\alpha$ form a decreasing sequence of closed subsets of $X$, we can fix $\delta<\omega_1$ such that $X_\alpha=X_\delta$ for every $\alpha\geq\delta$. If we had $X_\delta\neq\varnothing$, then we would have $X_{\delta+1}\subsetneq X_\delta$ by the definition of $\HC(X_\delta)$. Hence $X_\delta=\varnothing$, which implies
$$
X=\bigcup_{\alpha<\delta}\HC(X_\alpha).
$$
Since $\HC(X_\alpha)$ is strongly homogeneous for every $\alpha$ by Lemma \ref{lemma_homogeneous_clopen}, the proof is concluded.
\end{proof}

We conclude this section by pointing out that \cite[Theorem 2.4]{van_engelen_miller_steel} follows very easily from Lemma \ref{lemma_homogeneous_clopen}.

\begin{theorem}[van Engelen, Miller, Steel]\label{theorem_no_rigid}
Assume $\AD$. Then there exist no zero-dimensional rigid space.
\end{theorem}
\begin{proof}
Let $X$ be a zero-dimensional space such that $|X|\geq 2$, and assume without loss of generality that $X$ is a subspace of $2^\omega$. Set $U=\HC(X)$ and $V=\HC(X\setminus U)$, and observe that $U$ and $V$ are strongly homogeneous clopen subspaces of $X$ by Lemma \ref{lemma_homogeneous_clopen}. It is clear that $X$ is not rigid if $|U|\geq 2$ or $|V|\geq 2$, so assume that $|U|\leq 1$ and $|V|\leq 1$. Since $|X|\geq 2$, it follows that $|U|=|V|=1$. Therefore $X$ has at least two isolated points, which also implies that $X$ is not rigid.
\end{proof}

\section{Hereditarily rigid spaces}\label{section_hereditarily_rigid}

Throughout this section, we will be working in $\ZFC$. We begin by introducing two strengthenings of the standard notion of rigidity that will appear naturally in the remainder of this article. Usually, given a topological property $\PP$, the hereditarily $\PP$ spaces are those whose all subspaces satisfy $\PP$. In the case $\PP=\text{rigidity}$, following this approach too strictly only leads to trivialities. However, as in the following definition, a small tweak is sufficient to yield two interesting properties.

\begin{definition}
A space $X$ is \emph{$\cccc$-hereditarily rigid} if $X$ is $\cccc$-crowded and every $\cccc$-crowded subspace of $X$ is rigid. A space $X$ is \emph{strongly $\cccc$-hereditarily rigid} if $X$ is $\cccc$-crowded and the only homeomorphisms between $\cccc$-crowded subspaces of $X$ are of the form $\id_S$ for some $\cccc$-crowded subspace $S$ of $X$.
\end{definition}

It is clear that every $\cccc$-hereditarily rigid space is rigid, and that every strongly $\cccc$-hereditarily rigid space is $\cccc$-hereditarily rigid (see also Corollary \ref{corollary_rigid_not_hereditarily} and Question \ref{question_hereditarily_rigid_not_strongly}). The existence in $\ZFC$ of a strongly $\cccc$-hereditarily rigid space is the content of the next section. Proposition \ref{proposition_hereditarily_rigid_not_sigma-homogeneous} is the reason why the above notions are relevant in our context. It will easily follow from Proposition \ref{proposition_kappa-crowded}, which can be safely assumed to be folklore.

\begin{proposition}\label{proposition_kappa-crowded}
Let $\kappa$ be an uncountable cardinal. Then every homogeneous space of size $\kappa$ is $\kappa$-crowded.	
\end{proposition}
\begin{proof}
Let $X$ be a homogeneous space such that $|X|=\kappa$. Assume, in order to get a contradiction, that $U$ is a non-empty open subset of $X$ such that $|U|<\kappa$. Set
$$
\UU=\{h[U]:h\text{ is a homeomorphism of }X\},
$$
and observe that $\UU$ is a cover of $X$ by homogeneity. By considering a countable subcover of $\UU$, one sees that $|X|\leq |U|\cdot\omega<\kappa$, which is a contradiction.
\end{proof}

\begin{proposition}\label{proposition_hereditarily_rigid_not_sigma-homogeneous}
Let $X$ be a $\cccc$-hereditarily rigid space. Then $X$ is not $\sigma$-homogeneous.
\end{proposition}
\begin{proof}
Assume, in order to get a contradiction, that $X=\bigcup_{n\in\omega}X_n$, where each $X_n$ is homogeneous. Fix $n$ such that $|X_n|=\cccc$, and observe that $X_n$ is $\cccc$-crowded by Proposition \ref{proposition_kappa-crowded}. It follows that $X_n$ is rigid, which is a contradiction.
\end{proof}

We conclude this section by investigating these notions a little further, although these observations will not be needed later on. Recall that a subset $X$ of $2^\omega$ is \emph{Bernstein} if $K\cap X\neq\varnothing$ and $K\cap (2^\omega\setminus X)\neq\varnothing$ for every perfect subset $K$ of $2^\omega$. The following result first appeared as \cite[Theorem 6]{medini_van_mill_zdomskyy_2016} (see also the proof of \cite[Theorem 5]{medini_van_mill_zdomskyy_2016} for the fact that $X$ is Bernstein).

\begin{theorem}[Medini, van Mill, Zdomskyy]\label{theorem_homogeneous_rigid_complement}
There exists a subspace $X$ of $2^\omega$ with the following properties, where $Y=2^\omega\setminus X$:
\begin{itemize}
\item $X$ is Bernstein,
\item $X$ is rigid,
\item $Y$ is homogeneous.	
\end{itemize}
\end{theorem}

\begin{corollary}\label{corollary_rigid_not_hereditarily}
There exists a zero-dimensional $\cccc$-crowded rigid space that is not $\cccc$-hereditarily rigid.
\end{corollary}
\begin{proof}
Let $X$ be the subspace of $2^\omega$ given by Theorem \ref{theorem_homogeneous_rigid_complement}, and set $Y=2^\omega\setminus X$. Using the homogeneity of $Y$, one can fix a homeomorphism $h:Y\longrightarrow Y$ that is not the identity. By Lavrentiev's Theorem (see \cite[Exercise 3.10]{kechris}), we can fix $G\in\bP^0_2(2^\omega)$ and a homeomorphism $\widetilde{h}:G\longrightarrow G$ such that $h\subseteq\widetilde{h}$.

Notice that $2^\omega\setminus G\subseteq X$ is countable because $X$ is Bernstein. It follows that $G\cap X$ is also a Bernstein set, and in particular it is $\cccc$-dense in $2^\omega$. At this point, it is easy to realize that $\widetilde{h}\re G\cap X$ is a homeomorphism of $G\cap X$ other than the identity.
\end{proof}

We remark that, assuming $\VL$, one can even obtain an analytic space as in the above corollary (see Theorem \ref{theorem_definable_strongly_hereditarily_rigid}). However, we do not know whether $\VL$ yields a coanalytic version of this counterexample (see Question \ref{question_definable_rigid}).

\section{A counterexample in $\ZFC$}\label{section_counterexample_zfc}

Throughout this section, we will be working in $\ZFC$. The following result gives  a space with the strongest of the rigidity properties that we previously defined. Corollary \ref{corollary_not_sigma-homogeneous} shows how this is related to main topic of this article. The proof of Theorem \ref{theorem_strongly_hereditarily_rigid} was inspired by \cite[Proposition 2.10]{medini_van_mill_zdomskyy_2018}, which shows that the space of $\omega_1$ Cohen reals is rigid. While this proof was discovered with the help of elementary submodels, we subsequently realized that their use can be comfortably avoided.

\begin{theorem}\label{theorem_strongly_hereditarily_rigid}
There exists a zero-dimensional strongly $\cccc$-hereditarily rigid space.
\end{theorem}
\begin{proof}
By passing to a $\cccc$-crowded subspace, it will be enough to construct a subspace $X$ of $2^\omega$ such that $|X|=\cccc$ and the only homeomorphisms between $\cccc$-crowded subspaces of $X$ are of the form $\id_S$ for some $\cccc$-crowded subspace $S$ of $X$.

Fix an enumeration $\{h_\alpha:\alpha<\cccc\}$ of all homeomorphisms $h_\alpha:G_\alpha\longrightarrow H_\alpha$, where $G_\alpha,H_\alpha\in\bP^0_2(2^\omega)$. Also assume that $h_0:2^\omega\longrightarrow 2^\omega$ is the identity. Using transfinite recursion, pick $x_\alpha\in 2^\omega$ for $\alpha<\cccc$ such that the following conditions are satisfied:
\begin{enumerate}
\item\label{theorem_strongly_hereditarily_rigid_condition_domain} $x_\alpha\neq h_\beta(x_\gamma)$ for every $\beta,\gamma<\alpha$ such that $x_\gamma\in G_\beta$,
\item\label{theorem_strongly_hereditarily_rigid_condition_range} $x_\alpha\neq h_\beta^{-1}(x_\gamma)$ for every $\beta,\gamma<\alpha$ such that $x_\gamma\in H_\beta$.
\end{enumerate}
In the end, set $X=\{x_\alpha:\alpha<\cccc\}$.

Notice that $x_\alpha\neq x_\beta$ whenever $\alpha\neq\beta$ by our choice of $h_0$. In particular, we see that $|X|=\cccc$. Now fix $\cccc$-crowded subspaces $S$ and $T$ of $X$, and let $h:S\longrightarrow T$ be a homeomorphism. By Lavrentiev's Theorem (see \cite[Theorem 3.9]{kechris}), we can fix $\delta<\cccc$ such that $h\subseteq h_\delta$. We will show that $h(x_\alpha)=x_\alpha$ for every $\alpha >\delta$ such that $x_\alpha\in S$. Since $S$ is $\cccc$-crowded, this will imply that $S=T$ and $h=\id_S$, concluding the proof.

So fix $\alpha >\delta$ such that $x_\alpha\in S$, and let $\beta<\cccc$ be such that $h(x_\alpha)=x_\beta$. We will prove that $\beta=\alpha$ by showing that every other case leads to a contradiction.

\noindent{\bf Case 1:} $\beta>\alpha$.

This would violate condition $(\ref{theorem_strongly_hereditarily_rigid_condition_domain})$ in the construction of $x_\beta$.

\noindent{\bf Case 2:} $\beta<\alpha$.

Since $x_\alpha=h_\delta^{-1}(x_\beta)$, this would violate condition $(\ref{theorem_strongly_hereditarily_rigid_condition_range})$ in the construction of~$x_\alpha$.
\end{proof}

\begin{corollary}\label{corollary_not_sigma-homogeneous}
There exists a zero-dimensional space $X$ that is not $\sigma$-homogeneous.
\end{corollary}
\begin{proof}
This follows from Theorem \ref{theorem_strongly_hereditarily_rigid} and Proposition \ref{proposition_hereditarily_rigid_not_sigma-homogeneous}.
\end{proof}

\section{Preliminaries on $\VL$}\label{section_preliminaries_v=l}

We will assume some familiarity with the basic theory of $\GL$ (see \cite{kunen}) and basic recursion theory (see \cite{odifreddi}). We will say that $S\subseteq\omega^\omega$ is \emph{cofinal in the Turing degrees} if for every $x\in\omega^\omega$ there exists $y\in S$ such that $x$ is recursive in $y$. Given $x\in\omega^\omega$, we will denote by $\omega_1^x$ the least ordinal that is not recursive in $x$. Furthermore, we will say that $x\in\omega^\omega$ is \emph{self-constructible} if $x\in\GL(\omega_1^x)$. Given $X\subseteq Z\times W$ and $z\in Z$, we will use the notation $X_z=\{w\in W:(z,w)\in X\}$ for the vertical section of $X$ at $z$.

Theorem \ref{theorem_zoltan} is essentially a restatement of \cite[Theorem 1.3]{vidnyanszky}. The only significant difference is that we added that $X$ consists of self-constructible reals, but this is clear from the proof (thanks to \cite[Lemma 3.2]{vidnyanszky}). While this fact is described in \cite[page 173]{vidnyanszky} as ``one of the weaknesses of the method'', here it will be a crucial ingredient in our arguments (see the proof of Lemma \ref{lemma_topological_consequences}).

Following \cite[Definition 1.2]{vidnyanszky}, given $F\subseteq M^{\leq\omega}\times B\times M$, where $M$ and $B$ are sets of size $\omega_1$, we will say that $X\subseteq M$ is \emph{compatible with $F$} if there exist enumerations $B=\{p_\alpha:\alpha<\omega_1\}$, $X=\{x_\alpha:\alpha<\omega_1\}$ and, for every $\alpha<\omega_1$, a sequence $A_\alpha\in M^{\leq\omega}$ that is an enumeration of $\{x_\beta:\beta<\alpha\}$ in type $\leq\omega$ such that $x_\alpha\in F_{(A_\alpha,p_\alpha)}$ for every $\alpha<\omega_1$. Intuitively, one should think of $A_\alpha$ as enumerating the portion of the desired set $X$ constructed before stage $\alpha$. The section $F_{(A_\alpha,p_\alpha)}$ consists of the admissible candidates to be added at stage $\alpha$, where $p_\alpha$ encodes the current condition to be satisfied.

\begin{theorem}[Vidny\'anszky]\label{theorem_zoltan}
Assume $\VL$. Let $M=\omega^\omega$, and let $B$ be an uncountable Borel space.\footnote{\,More generally, one could replace $\omega^\omega$ with any other uncountable Polish space with a natural notion of Turing reducibility.} Assume that $F\subseteq M^{\leq\omega}\times B\times M$ is coanalytic, and that for all $(A,p)\in M^{\leq\omega}\times B$ the section $F_{(A,p)}$ is cofinal in the Turing degrees. Then there exists a coanalytic $X\subseteq M$ consisting of self-constructible reals that is compatible with $F$.
\end{theorem}

Lemma \ref{lemma_construction} already gives our desired counterexample, although the verification will be carried out in the next section. We will need the following fact, which is an immediate consequence of \cite[Example 4.27]{mansfield_weitkamp}.

\begin{lemma}\label{lemma_relation}
The relation $\{(x,y)\in\omega^\omega\times\omega^\omega:\omega_1^x<\omega_1^y\}$ is coanalytic.
\end{lemma}

\begin{lemma}\label{lemma_construction}
Assume $\VL$. Then there exists $X\subseteq \omega^\omega$ that satisfies the following conditions:
\begin{itemize}
\item $X$ is coanalytic,
\item $X$ is $\cccc$-dense in $\omega^\omega$,
\item Every element of $X$ is self-constructible,
\item If $x,y\in X$ and	$x\neq y$ then $\omega_1^x\neq\omega_1^y$.
\end{itemize}
\end{lemma}
\begin{proof}
Our plan is to apply Theorem \ref{theorem_zoltan} with $B=\omega^{<\omega}\times 2^\omega$, where $\omega^{<\omega}$ has the discrete topology. The purpose of $B$ is simply to ensure that $X$ will be $\cccc$-dense in $M=\omega^\omega$. More precisely, the factor $\omega^{<\omega}$ will allow us to specify a basic clopen set, while the factor $2^\omega$ will guarantee that every basic clopen set is met $\cccc$-many times.

Let $\pi:B\longrightarrow\omega^{<\omega}$ denote the projection on the first coordinate, and set $R=\{(x,y)\in\omega^\omega\times\omega^\omega:\omega_1^x<\omega_1^y\}$. Define $F\subseteq M^{\leq\omega}\times B\times M$ by declaring
$$
(A,p,x)\in F\text{ iff }\big(x\in\Ne_{\pi(p)}\text{ and }\forall n\in\omega\, (A(n)\,R\,x)\big).
$$
Using Lemma \ref{lemma_relation}, it is straightforward to check that $F$ is coanalytic. To check that each $F_{(A,p)}$ is cofinal in the Turing degrees, pick $(A,p)\in M^{\leq\omega}\times B$ and $z\in M$. First, fix $x'\in M$ such that $\omega_1^{A(n)}<\omega_1^{x'}$ for each $n$, then let $x''\in M$ code $x'$ on the even coordinates and $z$ on the odd coordinates. Finally, set $x=s^{\frown}x''\re (\omega\setminus n)$, where $s=\pi(p):n\longrightarrow\omega$. Since finite modifications do not affect Turing-reducibility, it is clear that $x\in\Ne_s$ is as desired.
\end{proof}

\section{Definable counterexamples under $\VL$}\label{section_definable_counterexamples}

In this section, we will show that the set/recursion-theoretic properties of the space given by Lemma \ref{lemma_construction} have significant topological consequences. The definable counterexample that we promised in the introduction will immediately follow (see Theorem \ref{theorem_definable_not_sigma-homogeneous}). As the reader could certainly have guessed, we will say that a space $X$ is \emph{$\sigma$-homogeneous with Borel witnesses} if there exist homogeneous $X_n\in\Borel(X)$ for $n\in\omega$ such that $X=\bigcup_{n\in\omega}X_n$.  Similarly, one can define what \emph{$\sigma$-homogeneous with closed witnesses} means.

\begin{lemma}\label{lemma_topological_consequences}
Assume $\VL$. Let $X\subseteq\omega^\omega$ be such that the following conditions are satisfied:
\begin{enumerate}
\item\label{lemma_topological_consequences_coanalytic} $X$ is coanalytic,
\item\label{lemma_topological_consequences_dense} $X$ is $\cccc$-dense in $\omega^\omega$,
\item\label{lemma_topological_consequences_self-constructible} Every element of $X$ is self-constructible,
\item\label{lemma_topological_consequences_injective} If $x,y\in X$ and	$x\neq y$ then $\omega_1^x\neq\omega_1^y$.
\end{enumerate}
Set $Y=\omega^\omega\setminus X$. Then: 
\begin{itemize}
\item $X$ and $Y$ are $\cccc$-crowded,
\item $X$ is strongly $\cccc$-hereditarily rigid,
\item $X$ is not $\sigma$-homogeneous,
\item $Y$ is rigid but not $\cccc$-hereditarily rigid,
\item $Y$ is not $\sigma$-homogeneous with Borel witnesses.
\end{itemize}
\end{lemma}
\begin{proof}
The fact that $X$ is $\cccc$-crowded follows trivially from condition $(\ref{lemma_topological_consequences_dense})$. Next, we will show that the following condition holds:
\begin{itemize}
\item[$\circledast$] If $h:S\longrightarrow T$ is a homeomorphism without fixed points, where $S,T\subseteq\omega^\omega$, then $\{x\in S\cap X:h(x)\in X\}$ is countable.
\end{itemize}
Let $h$ be as above. By Lavrentiev's Theorem, we can fix $G,H\in\bP^0_2(\omega^\omega)$ and a homeomorphism $\widetilde{h}:G\longrightarrow H$ such that $h\subseteq\widetilde{h}$. Also fix $\delta<\omega_1$ such that $\widetilde{h}$ and $\widetilde{h}^{-1}$ are coded in $\GL(\delta)$. Assume, in order to get a contradiction, that $h(x)\in X$ for uncountably many $x\in S\cap X$. In particular, by condition $(\ref{lemma_topological_consequences_injective})$, we can fix $x\in S\cap X$ such that $h(x)\in X$ and $\omega_1^x\geq\delta$. Set $y=h(x)$. Since $x\in\GL(\omega_1^x)$ by condition $(\ref{lemma_topological_consequences_self-constructible})$, it follows that $y=\widetilde{h}(x)\in\GL(\omega_1^x)$, hence $\omega_1^y\leq\omega_1^x$. On the other hand, the same argument applied to $\widetilde{h}^{-1}$ shows that $\omega_1^x\leq\omega_1^y$. Therefore $\omega_1^x=\omega_1^y$, which implies $x=y$ by condition $(\ref{lemma_topological_consequences_injective})$. This contradicts the assumption that $h$ has no fixed points, showing that condition $\circledast$ holds.

In order to prove that $X$ is strongly $\cccc$-hereditarily rigid, fix $\cccc$-crowded subspaces $S$ and $T$ of $X$, and let $h:S\longrightarrow T$ be a homeomorphism.  If there existed $x\in S$ such that $h(x)\neq x$, then condition $\circledast$ would be easily contradicted. Therefore $S=T$ and $h=\id_S$.

The fact that $X$ is not $\sigma$-homogeneous now follows from Proposition \ref{proposition_hereditarily_rigid_not_sigma-homogeneous}. Observe that $X$ cannot contain copies of $2^\omega$ because it is $\cccc$-hereditarily rigid. By condition $(\ref{lemma_topological_consequences_coanalytic})$ plus standard arguments involving the property of Baire (see \cite[Theorem 21.6 and Proposition 8.26]{kechris}), it follows that $X$ is meager in $\omega^\omega$. Therefore $Y$ is comeager in $\omega^\omega$, hence $\cccc$-crowded. Since analytic sets have the perfect set property (see \cite[Theorem 29.1]{kechris}), one sees that $Y$ is not $\cccc$-hereditarily rigid.

Next, we will show that $Y$ is rigid. Let $h:Y\longrightarrow Y$ be a homeomorphism. By Lavrentiev's Theorem, we can fix $G\in\bP^0_2(\omega^\omega)$ and a homeomorphism $\widetilde{h}:G\longrightarrow G$ such that $h\subseteq\widetilde{h}$. Notice that $\omega^\omega\setminus G\subseteq X$ is countable because $X$ is $\cccc$-hereditarily rigid. By condition $(\ref{lemma_topological_consequences_dense})$, it follows that $G\cap X$ is $\cccc$-dense in $\omega^\omega$. This shows that $\widetilde{h}\re (G\cap X)$ is the identity by the $\cccc$-hereditary rigidity of $X$, hence $\widetilde{h}$ is the identity.

It remains to show that $Y$ is not $\sigma$-homogeneous with Borel witnesses. Assume, in order to get a contradiction, that $Y=\bigcup_{n\in\omega}Y_n$, where each $Y_n\in\Borel(Y)$ is homogeneous. Fix a countable base $\BB\subseteq\bD^0_1(\omega^\omega)$ for $\omega^\omega$. Denote by $I$ the collection of all $(U,n)\in\BB\times\omega$ such that there exists $h$ with the following properties:
\begin{itemize}
\item $h:G\longrightarrow H$ is a homeomorphism, where $G,H\in\bP^0_2(\omega^\omega)$,
\item $h$ has no fixed points,
\item $G\cap Y_n=U\cap Y_n$,
\item $h[G\cap Y_n]=H\cap Y_n$.
\end{itemize}
For every $i\in I$, fix $h_i:G_i\longrightarrow H_i$ as above. Given $i\in I$, we will denote by $n_i$ the unique element of $\omega$ such that $i=(U,n_i)$ for some $U\in\BB$. Define
$$
X'=\bigcap_{i\in I}\{x\in\omega^\omega:x\notin G_i\text{ or }(x\in G_i\text{ and }h_i(x)\in Y\setminus Y_{n_i})\},
$$
and observe that $X'$ is analytic because each $Y_n\in\Borel(Y)$. We claim that $X'\Delta X$ is countable. This will conclude the proof, because $X$ is not analytic. Assume that $x\in X'\cap Y_n$ for some $n\in\omega$. Then, using the homogeneity of $Y_n$ and Lavrentiev's Theorem, it is not hard to show that $Y_n=\{x\}$. This shows that $X'\setminus X$ is countable. On the other hand, given $i\in I$, using condition $\circledast$, it is easy to realize that $h_i(x)\in Y\setminus Y_{n_i}$ for all but countably many $x\in G_i\cap X$. Since $i$ was arbitrary, this shows that $X\setminus X'$ is countable.

Obviously, the previous paragraph also shows that $Y$ is not $\sigma$-homogeneous with closed witnesses. We remark that Proposition \ref{proposition_closed_witnesses_not_rigid} gives an alternative proof of this fact.
\end{proof}

Combining Lemmas \ref{lemma_construction} and \ref{lemma_topological_consequences} immediately yields the following three results. These results respectively concern $\sigma$-homogeneity, classical rigidity, and the notions of hereditary rigidity introduced in Section \ref{section_hereditarily_rigid}. Notice that Theorem \ref{theorem_definable_strongly_hereditarily_rigid} can be viewed as a sharper version of Theorem \ref{theorem_definable_rigid}. The latter first appeared as \cite[Theorem 2.6]{van_engelen_miller_steel}.

\begin{theorem}\label{theorem_definable_not_sigma-homogeneous}
Assume $\VL$. Then there exists a zero-dimensional coanalytic space that is not $\sigma$-homogeneous. Furthermore, there exists a zero-dimensional analytic space that is not $\sigma$-homogeneous with Borel witnesses.
\end{theorem}

\begin{theorem}[van Engelen, Miller, Steel]\label{theorem_definable_rigid}
Assume $\VL$. Then there exist both analytic and coanalytic examples of zero-dimensional rigid spaces.
\end{theorem}

\begin{theorem}\label{theorem_definable_strongly_hereditarily_rigid}
Assume $\VL$. Then there exists a zero-dimensional coanalytic space that is strongly $\cccc$-hereditarily rigid. Furthermore, there exists a zero-dimensional analytic space that is $\cccc$-crowded and rigid, but not $\cccc$-hereditarily rigid.
\end{theorem}

\section{Final remarks and open questions}

We begin by remarking that the results of Section \ref{section_main} are ``local'' in nature. To express this precisely, recall from \cite[Definition 3.1]{carroy_medini_muller_2022} that a \emph{nice topological pointclass} is a function\footnote{\,Here, the term ``function'' is an abuse of terminology, as each nice topological pointclass is a proper class.  Therefore, any theorem that mentions them is strictly speaking an infinite scheme. Nice topological pointclasses are simply a convenient expositional tool that allows one to simultaneously state the Borel, Projective, and full-Determinacy versions of a theorem.} $\bS$ that satisfies the following requirements:
\begin{itemize}
\item The domain of $\bS$ is the class of all spaces,
\item $\bS(Z)\subseteq\PP(Z)$ for every space $Z$,
\item $\bS(Z)$ is closed under complements and finite unions for every space $Z$,
\item $\Borel(Z)\subseteq\bS(Z)$ for every space $Z$,
\item If $f:Z\longrightarrow W$ is a Borel function and $B\in\bS(W)$ then $f^{-1}[B]\in\bS(Z)$,
\item For every space $Z$, if $j[Z]\in\bS(W)$ for some Borel space $W$ and embedding $j:Z\longrightarrow W$, then $j[Z]\in\bS(W)$ for every Borel space $W$ and embedding $j:Z\longrightarrow W$.
\end{itemize}

The following are the most important examples of nice topological pointclasses (this can be verified using \cite[Exercise 37.3]{kechris} and the methods of \cite[Section 4]{medini_zdomskyy}):
\begin{itemize}
 \item $\bS(Z)=\Borel(Z)$ for every space $Z$,
 \item $\bS(Z)=\bigcup_{1\leq n<\omega}\bS^1_n(Z)$ for every space $Z$,\footnote{\,Following \cite[page 315]{kechris}, we will say that $X\in\bS^1_n(Z)$ if there exist a Polish space $W$, an embedding $j:Z\longrightarrow W$ and $\widetilde{X}\in\bS^1_n(W)$ such that $j[X]=\widetilde{X}\cap j[Z]$.}
 \item $\bS(Z)=\PP(Z)$ for every space $Z$.
\end{itemize}

Since all the relevant results from \cite{carroy_medini_muller_2022} are formulated as below using nice topological pointclasses, and similar formulations could be given of Theorems \ref{theorem_separation} and \ref{theorem_steel}, as well as of Lemmas \ref{lemma_level_0} and \ref{lemma_good_implies_reasonably_closed}, it is straightforward to check that the arguments of Section \ref{section_main} actually yield the following result.

\begin{theorem}\label{theorem_main_localized}
Let $\bS$ be a nice topological pointclass, and assume that $\Det(\bS(\omega^\omega))$ holds. Then:
\begin{itemize}
\item Every $X\in\bS(2^\omega)$ is $\sigma$-homogeneous,
\item No $X\in\bS(2^\omega)$ is rigid.
\end{itemize}
\end{theorem}
Notice that Theorems \ref{theorem_main} and \ref{theorem_no_rigid} can be obtained by setting $\bS=\PP$ in the above result, while \cite[Theorem 1]{ostrovsky} and \cite[Theorem 5.1]{van_engelen_thesis} can be obtained by setting $\bS=\Borel$.

We conclude the article with several open questions. The following is the most pressing one, as it is the only obstacle to having a complete picture of $\sigma$-homogeneity in the realm of zero-dimensional spaces.\footnote{\,Notice however that we do have a complete picture in the case of $\sigma$-homogeneity with closed witnesses.} We conjecture that the answer is ``yes'', with analytic witnesses, and that this could be shown using ideas from the proof of Theorem \ref{theorem_steel}.

\begin{question}\label{question_analytic}
Is every zero-dimensional analytic space $\sigma$-homogeneous?	
\end{question}

One of the themes of this article is that, in the context of zero-dimensional spaces, there seems to be a close parallel between rigidity and the lack of $\sigma$-homogeneity (for example, under $\AD$, these notions are vacuously equivalent). Therefore, it is natural to ask whether there exists a $\ZFC$ example of a zero-dimensional $\sigma$-homogeneous rigid space. The following example, constructed in \cite{van_engelen_van_mill}, gives a particularly nice positive answer. Notice that every rigid subspace of $\RRR$ must be zero-dimensional, as it cannot contain any intervals.

\begin{theorem}[van Engelen, van Mill]\label{theorem_decomposition_rigid}
In $\ZFC$, there exist homogeneous subspaces $X_1$ and $X_2$ of $\RRR$ such that $X=X_1\cup X_2$ is rigid. Furthermore, the spaces $X_1$ and $X_2$ are disjoint and homeomorphic.
\end{theorem}

Since Theorem \ref{theorem_main} gives a seemingly stronger property than mere $\sigma$-homogeneity, one might wonder whether the stronger versions can be distinguished from the standard one. For example, is it possible to construct a $\ZFC$ example of a zero-dimensional $\sigma$-homogeneous space that is not $\sigma$-homogeneous with closed witnesses? Once again, the space $X$ given by Theorem \ref{theorem_decomposition_rigid} yields a positive answer, thanks to the following simple proposition. Notice that $X$ is Baire by \cite[Lemma 3.2]{van_engelen_van_mill}.

\begin{proposition}\label{proposition_closed_witnesses_not_rigid}
Let $X$ be a zero-dimensional non-meager space. If $X$ is $\sigma$-homogeneous with closed witnesses then $X$ is not rigid.	
\end{proposition}
\begin{proof}
Assume that $X=\bigcup_{n\in\omega}X_n$, where each $X_n$ is a closed homogeneous subspace of $X$. First assume that $X$ has no isolated points. Since $X$ is non-meager, we can fix $n\in\omega$ such that $X_n$ has non-empty interior. This means that there exists a non-empty clopen subset $U$ of $X$ such that $U\subseteq X_n$. Observe that $U$ is an infinite clopen subspace of the zero-dimensional homogeneous space $X_n$. It follows that $U$ is not rigid (in fact, it is homogeneous), hence $X$ is not rigid.

Finally, assume that $X$ has at least one isolated point. If $X$ has at least two isolated points then it is obviously not rigid, so assume that $X$ has exactly one isolated point $x$. If $|X|=1$ then $X$ is not rigid by definition, so assume that $|X|\geq 2$. Then, it is not hard to verify that the previous case applies to $X\setminus\{x\}$. It follows that $X\setminus\{x\}$ is not rigid, which easily implies that $X$ is not rigid.
\end{proof}

However, we were not able to answer the following question. As the reader could certainly have guessed, we will say that a space $X$ is \emph{$\sigma$-homogeneous with pairwise disjoint witnesses} if there exist homogeneous pairwise disjoint subspaces $X_n$ of $X$ for $n\in\omega$ such that $X=\bigcup_{n\in\omega}X_n$.

\begin{question}\label{question_disjoint}
Is there a $\ZFC$ example of a zero-dimensional $\sigma$-homogeneous space that is not $\sigma$-homogeneous with pairwise disjoint witnesses? At least under additional set-theoretic assumptions?
\end{question}

While Corollary \ref{corollary_rigid_not_hereditarily} shows that rigidity is strictly weaker than $\cccc$-hereditary rigidity in $\ZFC$, we do not know whether the latter notion can be distinguished from its strong version. Furthermore, we do not know whether it is possible to obtain a coanalytic version of the space given by this corollary. The following questions ask for these counterexamples. Notice that, in order to answer Question \ref{question_definable_rigid}, one could try to give a coanalytic version of the rigid space given by Theorem \ref{theorem_homogeneous_rigid_complement}.

\begin{question}\label{question_hereditarily_rigid_not_strongly}
Is there a $\ZFC$ example of a zero-dimensional $\cccc$-hereditarily rigid space that is not strongly $\cccc$-hereditarily rigid? At least under additional set-theoretic assumptions?
\end{question}

\begin{question}\label{question_definable_rigid}
Assuming $\VL$, is there a zero-dimensional coanalytic rigid space that is not $\cccc$-hereditarily rigid?
\end{question}


\begin{thebibliography}{99}

\bibitem[AHN]{andretta_hjorth_neeman}\textsc{A. Andretta, G. Hjorth, I. Neeman.} Effective cardinals of boldface pointclasses. \emph{J. Math. Log.} \textbf{7:1} (2007), 35--82.

\bibitem[AvM1]{arhangelskii_van_mill_2014}\textsc{A. V. Arhangel$'$ski{\u\i}, J. van Mill.} Topological homogeneity. In: \emph{Recent Progress in General Topology III.} Atlantis Press, 2014. 1--68.

\bibitem[AvM2]{arhangelskii_van_mill_2020_1}\textsc{A. V. Arhangel$'$ski{\u\i}, J. van Mill.} Splitting Tychonoff cubes into homeomorphic and homogeneous parts. \emph{Topology Appl.} \textbf{275} (2020), 107018.

\bibitem[AvM3]{arhangelskii_van_mill_2020_2}\textsc{A. V. Arhangel$'$ski{\u\i}, J. van Mill.} Covering Tychonoff cubes by topological groups. \emph{Topology Appl.} \textbf{281} (2020), 107189.

\bibitem[CMM1]{carroy_medini_muller_2020}\textsc{R. Carroy, A. Medini, S. M\"uller.} Every zero-dimensional homogeneous space is strongly homogeneous under determinacy. \emph{J. Math. Logic.} \textbf{20:3} (2020), 2050015.

\bibitem[CMM2]{carroy_medini_muller_2022}\textsc{R. Carroy, A. Medini, S. M\"uller.} Constructing Wadge classes. \emph{Bull. Symb. Log.} \textbf{28:2} (2022), 207–-257.

\bibitem[vD]{van_douwen}\textsc{E. K. van Douwen.} A compact space with a measure that knows which sets are homeomorphic. \emph{Adv. in Math.} \textbf{52:1} (1984), 1--33.

\bibitem[vE1]{van_engelen_1986}\textsc{F. van Engelen.} Homogeneous Borel sets. \emph{Proc. Am. Math. Soc.} \textbf{96} (1986), 673--682.

\bibitem[vE2]{van_engelen_thesis}\textsc{A. J. M. van Engelen.} \emph{Homogeneous zero-dimensional absolute Borel sets.} CWI Tract, 27. Stichting Mathematisch Centrum, Centrum voor Wiskunde en Informatica, Amsterdam, 1986. iv+133 pp. Available at \verb"http://repository.cwi.nl/".

\bibitem[vE3]{van_engelen_1994}\textsc{F. van Engelen.} On Borel ideals. \emph{Ann. Pure Appl. Logic} \textbf{70:2} (1994), 177--203.

\bibitem[vEvM]{van_engelen_van_mill}\textsc{F. van Engelen, J. van Mill.} Decompositions of rigid spaces. \emph{Proc. Am. Math. Soc.} \textbf{89:3} (1983), 533--536.

\bibitem[vEMS]{van_engelen_miller_steel}\textsc{F. van Engelen, A. W. Miller, J. Steel.} Rigid Borel sets and better quasiorder theory. \emph{Logic and combinatorics, Proc. AMS-IMS-SIAM Conf., Arcata/Calif. 1985, Contemp. Math.} \textbf{65} (1987), 199--222.

\bibitem[En]{engelking}\textsc{R. Engelking.} \emph{General topology.} Revised and completed edition. Sigma Series in Pure Mathematics, vol. 6. Heldermann Verlag, Berlin, 1989.

\bibitem[EKM]{erdos_kunen_mauldin}\textsc{P. Erd\H{o}s, K. Kunen, R. D. Mauldin.} Some additive properties of sets of real numbers. \emph{Fund. Math.} \textbf{113:3} (1981), 187--199.

\bibitem[Ke]{kechris}\textsc{A. S. Kechris.} \emph{Classical descriptive set theory.} Graduate Texts in Mathematics, 156. Springer-Verlag, New York, 1995. xviii+402 pp.

\bibitem[Ku]{kunen}\textsc{K. Kunen.} \emph{Set theory.} Studies in Logic (London), 34. College Publications, London, 2011. viii+401 pp.

\bibitem[Lo1]{louveau_1983}\textsc{A. Louveau.} Some results in the Wadge hierarchy of Borel sets. In: \emph{Wadge degrees and projective ordinals. The Cabal Seminar. Volume II}, 3--23, Lect. Notes Log. \textbf{37}, Assoc. Symbol. Logic, La Jolla, CA, 2012. Originally published in: \emph{Cabal seminar 79-81}, 28--55, Lecture Notes in Math. \textbf{1019}, Springer, Berlin, 1983. 

\bibitem[Lo2]{louveau_book}\textsc{A. Louveau.} \emph{Effective Descriptive Set Theory.} Unpublished manuscript.

\bibitem[LSR1]{louveau_saint-raymond_1988_1}\textsc{A. Louveau, J. Saint-Raymond.} Les propri\'et\'es de r\'eduction et de norme pour les classes de Bor\'eliens. \emph{Fund. Math.} \textbf{131:3} (1988), 223--243.

\bibitem[LSR2]{louveau_saint-raymond_1988_2}\textsc{A. Louveau, J. Saint-Raymond.} The strength of Borel Wadge determinacy. In: \emph{Wadge degrees and projective ordinals. The Cabal Seminar. Volume II}, 74--101, Lect. Notes Log. \textbf{37}, Assoc. Symbol. Logic, La Jolla, CA, 2012. Originally published in: \emph{Cabal seminar 81-85}, 1--30, Lecture Notes in Math. \textbf{1333}, Springer, Berlin, 1988. 

\bibitem[MW]{mansfield_weitkamp}\textsc{R. Mansfield, G. Weitkamp.} \emph{Recursive aspects of descriptive set theory.} Oxford Logic Guides, 11. The Clarendon Press, Oxford University Press, New York, 1985. vii+144 pp.

\bibitem[Me1]{medini_2011}\textsc{A. Medini.} Products and h-homogeneity. \emph{Topology Appl.} \textbf{158:18} (2011), 2520--2527.

\bibitem[Me2]{medini_thesis}\textsc{A. Medini.} \emph{The topology of ultrafilters as subspaces of the Cantor set and other topics.} Ph.D. Thesis. University of Wisconsin - Madison. ProQuest LLC, Ann Arbor, MI, 2013. 110 pp. Available at \verb"http://www.math.wisc.edu/~lempp/theses/medini.pdf".

\bibitem[Me3]{medini_2019}\textsc{A. Medini.} On Borel semifilters. \emph{Topology Proc.} \textbf{53} (2019), 97--122.

\bibitem[MvMZ1]{medini_van_mill_zdomskyy_2016}\textsc{A. Medini, J. van Mill, L. Zdomskyy.} A homogeneous space whose complement is rigid. \emph{Israel J. Math.} \textbf{214:2} (2016), 583--595.

\bibitem[MvMZ2]{medini_van_mill_zdomskyy_2018}\textsc{A. Medini, J. van Mill, L. Zdomskyy.} Infinite powers and Cohen reals. \emph{Canad. Math. Bull.} \textbf{61:4} (2018), 812--821.

\bibitem[MZ]{medini_zdomskyy}\textsc{A. Medini, L. Zdomskyy.} Between Polish and completely Baire. \emph{Arch. Math. Logic.} \textbf{54:1-2} (2015), 231--245.

\bibitem[vM]{van_mill}\textsc{J. van Mill.} Homogeneous subsets of the real line. \emph{Compositio Math.} \textbf{46:1} (1982), 3--13.

\bibitem[Mi]{miller}\textsc{A. W. Miller.} Infinite combinatorics and definability. \emph{Ann. Pure Appl. Logic} \textbf{41:2} (1989), 179--203.

\bibitem[Od]{odifreddi}\textsc{P. Odifreddi.} \emph{Classical recursion theory. The theory of functions and sets of natural numbers.} Studies in Logic and the Foundations of Mathematics, 125. North-Holland Publishing Co., Amsterdam, 1989. xviii+668 pp.

\bibitem[Os]{ostrovsky}\textsc{A. Ostrovsky.} $\sigma$-homogeneity of Borel sets. \emph{Arch. Math. Logic.} \textbf{50:5-6} (2011), 661--664.

\bibitem[St1]{steel_1980}\textsc{J. R. Steel.} Analytic sets and Borel isomorphisms. \emph{Fund. Math.} \textbf{108:2} (1980), 83--88.

\bibitem[St2]{steel_1981}\textsc{J. R. Steel.} Determinateness and the separation property. \emph{J. Symbolic Logic.} \textbf{46:1} (1981), 41--44.

\bibitem[Te]{terada}\textsc{T. Terada.} Spaces whose all nonempty clopen subsets are homeomorphic. \emph{Yokohama Math. Jour.} \textbf{40} (1993), 87--93.

\bibitem[VW]{van_wesep}\textsc{R. A. Van Wesep.} Separation principles and the axiom of determinateness. \emph{J. Symbolic Logic.} \textbf{43:1} (1978), 77--81.

\bibitem[Vi]{vidnyanszky}\textsc{Z. Vidny\'anszky.} Transfinite inductions producing coanalytic sets. \emph{Fund. Math.} \textbf{224:2} (2014), 155--174.

\bibitem[Wa1]{wadge_thesis}\textsc{W. W. Wadge.} \emph{Reducibility and Determinateness on the Baire Space.} Ph.D. Thesis. University of California, Berkeley. ProQuest LLC, Ann Arbor, MI, 1983. 334 pp.

\bibitem[Wa2]{wadge_2012}\textsc{W. W. Wadge.} Early investigations of the degrees of Borel sets. In: \emph{Wadge degrees and projective ordinals. The Cabal Seminar. Volume II}, 3--23, Lect. Notes Log. \textbf{37}, Assoc. Symbol. Logic, La Jolla, CA, 2012.

\end{thebibliography}
\end{document}